\documentclass[11pt, english, draft]{article}
\usepackage{amssymb,amsmath}
\title{\bf More on the Non-Solvable Graphs and Solvabilizers}
\author{ {\bf B. Akbari} \\[0.2cm]
{\em Department of Mathematics, Sahand
University of Technology,}\\
 {\em P.O. Box 51335-1996, Tabriz, IRAN.}\\[0.2cm]
{\em E-mails}: {\tt b.akbari@sut.ac.ir} and {\tt
b.akbari@dena.kntu.ac.ir}}

\newenvironment{proof}{\noindent {\em {Proof}}.}{$\square$
\medskip}
\newtheorem{definition}{Definition}[section]
\newtheorem{corollary}{Corollary}[section]

\newtheorem{theorem}{Theorem}[section]

\newtheorem{lm}{Lemma}[section]

\newtheorem{problem}{Problem}[section]
\begin{document}
\maketitle
\begin{abstract}

\noindent The non-solvable graph of a finite group $G$, denoted by
${\cal S}_{G}$, is a simple graph whose vertices are the elements
of $G$ and there is an edge between $x, y\in G$ if and only if
$\langle x, y\rangle$ is not solvable. If $R(G)$ is the solvable
radical of $G$, the isolated vertices in ${\cal S}_{G}$ are
exactly the elements of $R(G)$.  So in the case where $G$ is a
non-solvable group, it is wise to consider the induced subgraph of
${\cal S}_{G}$ with respect to $G\setminus R(G)$. This graph is
denoted by $\widehat{{\cal S}_G}$. Let $G$ be a finite group and
$x\in G$. The solvabilizer of $x$ with respect to $G$, denoted by
$Sol_G(x)$, is the set $\{y\in G\ |\ \langle x, y\rangle \ {\rm
is\ solvable}\}$. The purpose of this paper is to study some
properties of the non-solvable graph $\widehat{{\cal S}_G}$ and
the structure of $Sol_G(x)$ for every $x\in G$. We also show that
there is no finite group in which some vertices in its
non-solvable graph have the degree $n-2$ where $n=|G|-|R(G)|$.
\end{abstract}
\renewcommand{\baselinestretch}{1.1}
\def\thefootnote{ \ }
\footnotetext{{\em $2010$ Mathematics Subject Classification}:
20D05, 20D06, 20D08, 20D10, 05C25.}
\section{Introduction}
All groups appearing here are assumed to be finite. One of the
most important methods and interesting subjects is to study finite
groups by algebraic properties associated with them. There are a
lot of ways to relate an algebraic property to a finite group. One
of them is to consider some properties of the graphs associated
with it. Each property of the graph can learn us a property of the
group. Let $G$ be a finite group. The {\em non-solvable graph}
${\cal S}_{G}$ is a simple graph that constructs as follows. The
vertex set is $G$ and two distinct elements $x$ and $y$ are
adjacent if and only if the subgroup $\langle x, y\rangle$ is not
solvable.

In fact, Thompson's Theorem asserts that a group $G$ is solvable
if and only if $\langle x, y\rangle$ is solvable for every $x,
y\in G$. Hence, ${\cal S}_{G}$ is an empty graph if and only if
$G$ is solvable. Therefore, we only study ${\cal S}_{G}$ if $G$ is
not solvable.

For two non-empty subsets $A, B$ of $G$, we call $Sol_A(B)$ the
{\em solvabilizer} of $B$ with respect to $A$ which is the subset
$$\{a\in A\ |\ \langle a, b\rangle \ {\rm is\ solvable}\ \forall b\in B\}.$$
Note that $Sol_A(B)$ is not necessarily a subgroup of $G$. We put
$Sol_A(x):=Sol_A(\{x\})$ and $Sol(G):=Sol_G(G)$. Let $R(G)$ be the
solvable radical of $G$. In \cite{guralni}, it was obtained that
$Sol(G)=R(G)$. It is also clear that $Sol_A(x)=Sol_A(\langle
x\rangle)$. We focus our attention on $Sol_G(x)$.

It was shown in \cite{doron} that $Sol_G(x)$ is the union of all
solvable subgroups of $G$ containing $x$. It was also proved that
$Sol_G(x)$ is a disjoint union of some cosets of $ \langle
x\rangle$.

According to above, for every $x\in G$ we have
$${\rm deg}(x)=|G|-|Sol_G(x)|,$$ where ${\rm deg}(x)$ is the
degree of vertex $x$ in ${\cal S}_{G}$.

It is obvious that every two elements of $Sol(G)$ are not adjacent
in ${\cal S}_{G}$. On the other hand, as mentioned before,
$Sol(G)=R(G)$ where $R(G)$ is the solvable radical of $G$, which
means that if $x$ is an element of $G$ such that for every $y\in
G$, $\langle x, y\rangle$ is solvable, then $x\in R(G)$.
Therefore, for all $x\in G\setminus R(G)$, there exists an element
$y\in G\setminus R(G)$ such that $\langle x, y\rangle$ is not
solvable. So we can conclude that the elements of $R(G)$ are
exactly the isolated vertices in ${\cal S}_{G}$. Hence, if $G$ is
a non-solvable group, then it is logical to consider the induced
graph of ${\cal S}_{G}$ with respect to $G\setminus Sol(G)$ which
is denoted by $\widehat{{\cal S}_G}$. It is seen that the degree
of vertex $x\in G\setminus Sol(G)$ in ${\cal S}_{G}$ is equal to
its degree in $\widehat{{\cal S}_G}$.

The non-solvable graph of a group can be generalized in the
following way (see \cite{abdollahi}).

Let $G$ be a finite group. The non-nilpotent graph of $G$, which
is denoted by ${\cal N}_{G}$, is a simple graph whose vertices are
the elements of $G$ and two vertices $x, y$ are adjacent by an
edge if and only if $\langle x, y\rangle$ is not nilpotent. The
induced subgraph of ${\cal N}_{G}$ on $G\setminus nil(G)$, where
$nil(G)=\{ x\in G| \ \langle x, y\rangle {\rm \ is \ nilpotent \
for \ all} \ y\in G\}$, was introduced as $\widehat{{\cal N}_G}$.
This graph was completely verified in \cite{abdollahi}. Clearly,
${\cal S}_{G}$ (resp. $\widehat{{\cal S}_G}$) is a subgraph of
${\cal N}_{G}$ (resp. $\widehat{{\cal N}_G}$).

We are going to focus on non-solvable graph. In fact, we are
interested in finding the structure of a group through some
properties of its non-solvable graph. Many properties of this
graph were studied in \cite{doron}.

Our terminology and notation for groups will be standard. Thus, we
only introduce some notation for graphs used in this paper.

{\em Notation for Graphs.} A {\em simple graph} $\Gamma$ with
vertex set $V=V(\Gamma)$ and edge set $E=E(\Gamma)$ is a graph
with no loops or multiple edges. A graph $\Gamma$ is $k$-regular
if the degrees of all vertices in $\Gamma$ is $k$. A {\em regular
graph} is one that is $k$-regular for some $k$. A $(n-1)$-regular
graph with $n$ vertices is said a {\em complete graph}. A complete
graph with $n$ vertices is denoted by $K_n$. A set of vertices of
a graph is independent if the vertices are pairwise nonadjacent.
The {\em independence number} $\alpha(\Gamma)$ of a graph $\Gamma$
is the cardinality of a largest independent set of $\Gamma$. The
distance between two vertices of a graph is the minimum length of
the paths connecting them. The {\em diameter} of a graph is the
greatest distance between two vertices of the graph. An acyclic
graph is one that contains no cycles. A connected acyclic graph is
called a {\em tree}. A graph is {\em bipartite} if its vertex set
can be partitioned into two subsets $X$ and $Y$ so that every edge
has one end in $X$ and one end in $Y$, such a partition $(X, Y)$
is called a {\em bipartition} of the graph, and $X$ and $Y$ its
{\em parts}. We recall that a graph is bipartite if and only if it
contains no odd cycle (a cycle of odd length). A bipartite graph
with bipartition $(X, Y)$ in which every two vertices from $X$ and
$Y$ are adjacent is called a {\em complete bipartite graph} and
denoted by $K_{|X|, |Y|}$. An edge subdivision operation for an
edge $e\in E$ with two endpoints $u, v\in V$, is the deletion of
$e=uv$ from $\Gamma$ and the addition of two edges $uw$ and $wv$
along with the new vertex $w$. A graph which has been derived from
$\Gamma$ by a sequence of edge subdivision operations is called a
subdivision of $\Gamma$. A graph is {\em planar} if it can be
drawn on the plane without edges crossing except at endpoints. In
fact, a graph is planar if and only if it dose not contain a
subdivision of $K_{3, 3}$ and $K_5$(Kuratowski's Theorem).

In this paper, we are interested in characterizing certain
properties of a group in terms of some properties of  non-solvable graph
and solvabilizers.

In section 3, we study some properties of the induced subgraph
${\cal S}_{G}$ with respect to $G\setminus Sol(G)$($\widehat{{\cal
S}_G}$), for a group $G$.

In section 4, we investigate the structure of the solvabilizer of
$x\in G$ with respect to $G$. We prove that $N_G(\langle
x\rangle)\subseteq Sol_G(x)$ where $N_G(\langle x\rangle)$ is the
normalizer of $\langle x\rangle$ in $G$. In general, we show that
if $H$ is a solvable subgroup of $G$, then for all $x\in H$,
$N_G(H)\subseteq Sol_G(x)$ (Lemma $4. 1$). Furthermore, in
\cite{doron}, it was shown that for $x\in G$, $deg(x)\leqslant
n-2$ where $n=|G|-|Sol(G)|$. In this paper, we prove that
$deg(x)\neq n-2$ and so $deg(x)\leqslant n-3$ for every $x\in G$.
\section{Preliminary Results}
In this section, we express some results obtained in \cite{doron}
which will help us for further investigations. We begin with a
Theorem taken from \cite{guralni}.
\begin{theorem}\label{connected}
Let $G$ be a non-solvable group and $x, y\in G$ such that $x,
y\notin Sol(G)$. Then there exists $z\in G$ such that $\langle x,
z\rangle$ and $\langle y, z\rangle$ are not solvable.
\end{theorem}
Note that in Theorem \ref{connected}, if $G$ is a non-solvable
group in which any proper subgroup is solvable (equivalently every
maximal subgroup is solvable), then it has the property that $G$
is simple modulo the Frattini subgroup (trivial) and is generated
by two elements. The simple groups occurring are classified by
Thompson (This is a famous result of John Thompson in series of
papers on N-groups). So the Theorem still holds. For example, let
G be the alternating group $A_5$ whose the maximal subgroups are
as follows: the alternating group $A_4$, the dihedral group
$D_{10}$ and the symmetric group $S_3$ (see \cite{atlas}). Then
using the fact that any finite simple group can be generated by
two elements (Steinberg in \cite{steinberg}  proved this for a
Chevalley group and Aschbacher and Guralnick in
\cite{aschbacher-gu} verified it for the sporadic groups), any one
nontrivial element is part of a generating pair and the Theorem
still holds.

As a straightforward result of Theorem \ref{connected}, we have
the non-solvable graph $\widehat{{\cal S}_G}$ is connected and its
diameter is at most $2$. More precisely, it was shown  in
\cite{doron} that the diameter of $\widehat{{\cal S}_G}$ can not
be equal to $1$. So we can state the following Lemma.
\begin{lm}{\rm (\cite{doron})}
Let $G$ be a non-solvable group. Then $diam(\widehat{{\cal
S}_G})=2$.
\end{lm}
If $x, y$ are two elements of a group $G$ with order $2$, then
$\langle x, y\rangle$ is a dihedral group. Indeed, we have the
following Lemma.
\begin{lm}{\rm (\cite{doron})}\label{o2}
Let $G$ be a group and $x, y\in G$ such that $o(x)=o(y)=2$. Then
$\langle x, y\rangle$ is solvable.
\end{lm}
Now, we collect some results on solvabilizer of an element of $G$
with respect to $G$ obtained in \cite{doron}.
\begin{lm}{\rm (\cite{doron})}\label{sol1}
Let $G$ be a group. Suppose that $N\lhd G$ such that $N\subseteq
Sol(G)$ and $x, g\in G$. Then the following statements hold:
\begin{itemize}
\item[$(1)$] $Sol_{G/N}(xN)=Sol_G(x)/N$; \item[$(2)$]
$Sol_G(gxg^{-1})=gSol_G(x)g^{-1}$; \item[$(3)$] If $A, B\subseteq
G$ are two subsets such that $A\subseteq B$ and $x\in A$ is an
element, then $Sol_A(x)\subseteq Sol_B(x)$.
\end{itemize}
\end{lm}
\begin{lm}{\rm (\cite{doron})}\label{sol}
Let $G$ be a group and $x\in G$. Then we have:
\begin{itemize}
\item[$(1)$] $|Sol_G(x)|$ is divisible by $|Sol(G)|$; \item[$(2)$]
$|Sol_G(x)|$ is divisible by $o(x)$ and $|C_G(x)|$.
\end{itemize}
\end{lm}
As mentioned before, ${\rm deg}(x)=|G|-| Sol_G(x)|$ for every
$x\in G$. Thus, it is found form Lemma \ref{sol} $(2)$ that
$|C_G(x)|\mid deg(x)$.
\begin{lm}
Let $G$ be a non-solvable group and $x \in G$. Moreover, let $H$
be a solvable subgroup of $G$. Then the following statements hold:
\begin{itemize}
\item[$(1)$] $Sol_G(x)=Sol_G(x^i)$ where $1\leqslant i\leqslant
o(x)$ and $(i, o(x))=1$. In particular, $deg(x)=deg(x^i)$.
\item[$(2)$] $H\leqslant Sol_G(x)$, for every $x\in H$.
\end{itemize}
\end{lm}
\begin{proof}
Part $(1)$ is a straightforward result of Lemma $2. 11$ in
\cite{doron}. Also, part $(2)$ is a conclusion of Thompson's
Theorem.
\end{proof}

In \cite{doron}, the degrees of vertices in the non-solvable graph
were investigated and the following results were found.
\begin{lm}{\rm (\cite{doron})}\label{order&degree}
Let $G$ be a non-solvable group and $x\in G\setminus Sol(G)$.
Moreover, assume that $n=|G|-|Sol(G)|$. Then the following hold:
\begin{itemize}
 \item[$(1)$] $2o(x)\leqslant deg(x)$;
 \item[$(2)$] $5<deg(x)<n-1$;
 \item[$(3)$] $deg(x)$ is not a prime.
\end{itemize}
\end{lm}
In section 4, we will show that if $G$ is a non-solvable group,
then for all $x\in G$, $deg(x)\neq n-2$. Thus, we can conclude
that $deg(x)\leqslant n-3$ for all $x\in G$.

As mentioned before, for an element $x\in G$, $Sol_G(x)$ needs not
to be a subgroup of $G$ in general. If $G$ is a group in which
$Sol_G(x)\leqslant G$ for all $x\in G$, then $G$ is called an
$S$-group. The structure of an $S$-group was studied in
\cite{doron}. In fact, the following Lemma was proved.
\begin{lm}{\rm (\cite{doron})}\label{sgroup}
Let $G$ be a group. Then $G$ is solvable if and only if $G$ is an
$S$-group.
\end{lm}
The result obtained in Lemma \ref{sgroup}, states an equivalent
condition for solvability. In other words, $G$ is solvable if and
only if $G$ has the following property: For every $x, y, z\in G$,
if $\langle x, y\rangle$ and $\langle x, z\rangle$ are solvable,
then $\langle x, yz\rangle$ is solvable.

We recall that if a graph contains the complete bipartite graph
$K_{3, 3}$, then it is not planar. In \cite{doron}, the following
Lemma was proved.
\begin{lm}{\rm (\cite{doron})}\label{k4,4}
Let $G$ be a non-solvable group. Then $\widehat{{\cal S}_G}$
contains $K_{4, 4}$ as a subgraph.
\end{lm}
Now as a conclusion of Lemma \ref{k4,4} and Kuratowski's Theorem,
we can see that $\widehat{{\cal S}_G}$ is not planar.
\begin{lm}{\rm (\cite{doron})}
Let $G$ be a non-solvable group. Then $\widehat{{\cal S}_G}$ is
irregular.
\end{lm}

\section{Some Properties of Non-solvable Graphs}
In this section, we are going to investigate some graphic
properties of $\widehat{{\cal S}_G}$.
\begin{lm}
Let $G$ be a non-solvable group. Then $\widehat{{\cal S}_G}$ is
not a tree.
\end{lm}
\begin{proof}
Assume to the contrary that $\widehat{{\cal S}_G}$ is a tree. Then
it contains at least two vertices having degree one (see
\cite{bondy}) which contradicts Lemma \ref{order&degree} $(2)$.
So, $\widehat{{\cal S}_G}$ is not a tree.
\end{proof}

The elements of $Sol(G)$ are exactly the isolated vertices in
${\cal S}_G$. Thus $Sol(G)$ is an independent set of ${\cal S}_G$
and so the independence number of ${\cal S}_G$ is greater or equal
than $|Sol(G)|$. We can also state the following Lemma.
\begin{lm}
Let $G$ be a non-solvable group. Then $\alpha({\cal S}_G)\geqslant
max\{o(x)| x\in G\}$.
\end{lm}
\begin{proof}
For every element $x\in G$, the set $\langle x\rangle$ is an
independent set because it is clear that for all $1\leqslant i,
j\leqslant o(x)$, $\langle x^i, x^j\rangle\leqslant\langle
x\rangle$ is a solvable subgroup of $G$ and thus $x^i$ and $x^j$
are not adjacent in ${\cal S}_G$. So the proof is complete.
\end{proof}

Let $A$ be an independent set of $\widehat{{\cal S}_G}$. Then it
is easy to see that $A\cup Sol(G)\subseteq Sol_G(x)$ for all $x\in
A$. Moreover, if $A\cup \{1\}$ is a subgroup of $G$, then it is a
solvable subgroup.

Now, we consider the non-solvable graphs of subgroups and quotient
groups of a finite group.
\begin{lm}\label{subgraph}
Let $G$ be a non-solvable group. Let $H$ and $N$ be two subgroups
of $G$ such that $N\lhd G$, $N\subseteq Sol(G)$ and
$Sol(G)\subseteq H$. Then the following statements hold:
\begin{itemize}
 \item[$(1)$] If $x$ and $y$ are joined in $\widehat{{\cal S}_H}$
 for every $x, y\in H$, then $x$ and $y$ are joined in
 $\widehat{{\cal S}_G}$. In other words,
 $\widehat{{\cal S}_H}$ is a subgraph of $\widehat{{\cal S}_G}$.
 \item[$(2)$] For two elements $x, y\notin
Sol(G)$, $xN$ and $yN$ are adjacent in $\widehat{{\cal S}_{G/N}}$
if and only if $x$ and $y$ are adjacent in $\widehat{{\cal S}_G}$.
\end{itemize}
\end{lm}
\begin{proof}
$(1)$ Since $Sol(G)\subseteq H$, so it is seen that
$Sol(G)\leqslant Sol(H)$. Thus if $x, y\notin Sol(H)$, then $x,
y\notin Sol(G)$. The rest of proof is clear.

$(2)$ We only prove the sufficiency. The necessity is similar.
Assume that $xN$ and $yN$ are adjacent in $\widehat{{\cal
S}_{G/N}}$. It follows that $yN\notin Sol_{G/N}(xN)$. Then
considering Lemma \ref{sol1} $(1)$, we obtain that
$$yN\notin Sol_G(x)/N.$$ Hence $y\notin Sol_G(x)$ which implies
that $x$ and $y$ are adjacent in $\widehat{{\cal S}_G}$.
\end{proof}

In a view of Lemma \ref{subgraph}, if $Sol(G)=1$, then the
non-solvable graph of each subgroup of $G$ is a subgraph of
$\widehat{{\cal S}_G}$.
\begin{corollary}
Let $G$ be a non-abelian simple group and $H$ a subgroup of $G$.
Then $\widehat{{\cal S}_H}$ is a subgraph of $\widehat{{\cal
S}_G}$.
\end{corollary}
\begin{proof}
The proof is obvious.
\end{proof}
\begin{lm}
Let $G$ be a non-solvable group. Let $H$ be a proper subgroup of
$G$ and $N$ a normal subgroup of $G$ such that $N\subseteq Sol(G)$
and $Sol(G)\subseteq H$. Then the following statements hold:
\begin{itemize}
 \item[$(1)$] $\widehat{{\cal S}_H}$ is not isomorphic to
$\widehat{{\cal S}_G}$.
 \item[$(2)$] $\widehat{{\cal S}_{G/N}}$ is not isomorphic to
$\widehat{{\cal S}_G}$.
\end{itemize}
\end{lm}
\begin{proof}
$(1)$ By contrast, assume that $\widehat{{\cal S}_H}\cong
\widehat{{\cal S}_G}$.  Thus the vertex set of $\widehat{{\cal
S}_H}$ coincides with one of $\widehat{{\cal S}_G}$. It follows
that
$$|H|-|Sol(H)|=|G|-|Sol(G)|.$$ We observe that $Sol(G)\leqslant
Sol(H)$. If $Sol(G)=Sol(H)$, then $|G|=|H|$ which is impossible.
Hence, $Sol(G)\lvertneqq Sol(H)$. It implies that
$$|Sol(G)|\leqslant \frac{1}{2} |Sol(H)|.$$ So we can conclude
that
$$|G|\leqslant |H|-2|Sol(G)|+|Sol(G)|=|H|-|Sol(G)|,$$
which is false. Therefore, $\widehat{{\cal S}_H}$ is not
isomorphic to $\widehat{{\cal S}_G}$.

$(2)$ By contrast, suppose that $\widehat{{\cal S}_{G/N}}\cong
\widehat{{\cal S}_G}$. It forces that the vertex set of
$\widehat{{\cal S}_{G/N}}$ coincides with one of $\widehat{{\cal
S}_G}$. Thus we have
$$|G|-|Sol(G)|=\frac{|G|}{|N|}-\frac{|Sol(G)|}{|N|},$$
which is a contradiction. So the proof is complete.
\end{proof}
\section{The Structure of Solvabilizers and Non-solvable Graphs with Certain Degrees of Vertices}
In this section, we consider the structure of solvabilizer
$Sol_G(x)$ for every $x\in G$. We also study non-solvable graphs
whose some vertices have certain degree. Finally, we state a
problem on characterization of finite groups by solvabilizers.
\begin{theorem}\label{Ng(x)}
Let $G$ be a non-solvable group and $x$ an element of $G$. Then
$N_G(\langle x\rangle)\subseteq Sol_G(x)$. In particular, if $x,
y\in G\setminus Sol(G)$ are two elements such that $y\in
N_G(\langle x\rangle)$, then $y$ is not adjacent to $x$ in
$\widehat{{\cal S}_G}$.
\end{theorem}
\begin{proof}
Suppose that $y\in N_G(\langle x\rangle)$. Thus $\langle
y\rangle\subseteq N_G(\langle x\rangle)$ which yields that
$\langle x\rangle\langle y\rangle$ is a subgroup of $G$. It
follows that $\langle x, y\rangle=\langle x\rangle\langle
y\rangle$. Moreover, it is easy to see that $\langle x\rangle\lhd
\langle x, y\rangle$. We observe that $\langle x\rangle$ and
$\langle x, y\rangle/\langle x\rangle\cong\langle y\rangle$ are
solvable. So we can conclude that $\langle x, y\rangle$ is
solvable. Therefore, $y\in Sol_G(x)$. The rest of proof is
obvious.
\end{proof}

As an important result on solvable groups, Thompson's Theorem
states that a group $G$ is solvable if and only if $\langle x,
y\rangle$ is solvable for every $x, y\in G$. Indeed, Theorem
\ref{Ng(x)} confirms the following fact.
\begin{corollary}
Let $G$ be a group. If all cyclic subgroups of $G$ are normal
subgroups in $G$, then $G$ is solvable.
\end{corollary}
Before proceeding our study, we define certain subgroups of a
group.
\begin{definition}
Let $G$ be a group. A local subgroup of $G$ is a subgroup $K$ of
$G$ if there is a nontrivial solvable subgroup $H$ of $G$ such
that $K=N_G(H)$.
\end{definition}
When we are considering the solvibilizers of the elements
belonging to the solvable subgroups of a finite group, we can
generalize Theorem \ref{Ng(x)} to the following Lemma.
\begin{lm}
Let $G$ be a group and $K=N_G(H)$ a local subgroup of $G$ for some
solvable subgroup $H$ of $G$. Then for every $x\in H$, we have
$K\subseteq Sol_G(x)$.
\end{lm}
\begin{proof}
Suppose that $y\in K$. It is seen that $\langle y\rangle H$ is a
subgroup of $G$. We also have $$\frac{\langle y\rangle H}{H}\cong
\frac{\langle y\rangle}{H\cap\langle y\rangle}.$$ Since
$\frac{\langle y\rangle}{H\cap\langle y\rangle}$ and $H$ are
solvable, so $\langle y\rangle H$ is solvable. On the other hand,
we observe that for every $x\in H$, $\langle x, y\rangle\leqslant
\langle y\rangle H$. It follows that $\langle x, y\rangle$ is
solvable and hence $y\in Sol_G(x)$. Therefore, the proof is
complete.
\end{proof}


\begin{theorem}\label{n-2}
Let $G$ be a non-solvable group and $n=|G|-|Sol(G)|$. Then there
is no element $x\in G\setminus Sol(G)$ such that $deg(x)=n-2$.
\end{theorem}
\begin{proof}
Suppose to the contrary that $x$ an element of $G$ with
$deg(x)=n-2$.

It is good to mention that $deg(x)=|G|-|Sol_G(x)|$. Thus
$$|Sol_G(x)|=|Sol(G)|+2.$$
According to Lemma \ref{sol} $(1)$, $|Sol_G(x)|$ is divisible by
$|Sol(G)|$ which forces that $|Sol(G)|=1$ or $2$. First of all, we
claim that $|Sol(G)|\neq 2$.

Assume that $|Sol(G)|=2$. Then we can see that $|Sol_G(x)|=4$. We
find from Lemma \ref{sol} $(2)$ that $o(x)\mid |Sol_G(x)|$ which
follows that $o(x)=2$ or $4$. Now, we examine these cases
separately.

{\em Case $1$.} First let $o(x)=2$. If $x$ is the only element of
$G$ with $o(x)=2$, then for every $g\in G$, $o(g^{-1}xg)=o(x)$
which yields that $x=g^{-1}xg$. It implies that $x\in
Z(G)\subseteq Sol(G)$ which is false. Therefore, there exists
$y\in G$, $y\neq x$, with $o(y)=2$. According to Lemma \ref{o2},
$\langle x, y\rangle$ is solvable and thus $y\in Sol_G(x)$. So we
can conclude that $G$ has at most three elements of order $2$,
namely, $x$, $y_1$ and $y_2$.

We claim that there exists two elements $g_1, g_2\in G$ such that
$y_i=g_i^{-1}xg_i$. For this purpose, we assume to the contrary
that for every $g_1, g_2\in G$, $g_1^{-1}xg_1=g_2^{-1}xg_2$. It
implies that $g_1g_2^{-1}\in C_G(x)$. On the other hand, we have
$$C_G(x)\leqslant C_G(\langle x\rangle)\leqslant N_G(\langle x\rangle)\subseteq Sol_G(x).$$
Therefore, for any $g_1, g_2\in G$, $o(g_1g_2^{-1})=2$ or $4$. If
$o(g_1g_2^{-1})=2$ for all $g_1, g_2\in G$, then $G$ is an
elementary group which forces that $G$ is nilpotent. This is a
contradiction. Assume now that $o(g_1g_2^{-1})=4$ for all $g_1,
g_2\in G$. Then every nontrivial element of $G$ has order $2$ or
$4$. In \cite{Lytkina}, the structure of a group with elements of
order at most $4$ was completely determined. In fact, we use
Theorem $1$ in \cite{Lytkina} and gain a contradiction.
Consequently, there exist two elements $g_1, g_2\in G$ such that
$y_i=g_i^{-1}xg_i$.

Since $|Sol(G)|=2$, thus $Sol(G)=\langle y_1\rangle$ or
$Sol(G)=\langle y_2\rangle$. On the other hand, for every $g\in
G$, we have $g^{-1}Sol(G)g\subseteq Sol(G)$ because $Sol(G)\lhd
G$. It follows that $x\in Sol(G)$ which is false.

{\em Case $2$.} Let $o(x)=4$. 
We obtain from Lemma \ref{Ng(x)} that $$\langle x\rangle\subseteq
N_G(\langle x\rangle)\subseteq Sol_G(x).$$ Since $|\langle
x\rangle|=|Sol_G(x)|$, hence $$Sol_G(x)=\langle x\rangle=\{1, x,
x^2, x^3\}.$$ clearly, $x, x^3\notin Sol(G)$.

We claim that $x^2\notin Sol(G)$. Suppose to the contrary that
$x^2\in Sol(G)$. Then, according to the order of $Sol(G)$, we have
$Sol(G)=\langle x^2\rangle$. In the sequel, for the sake of
simplicity of the notation, we put $K:=Sol(G)$. Now, it follows
from Lemma \ref{sol1} that
$$Sol_{G/K}(xK)=\frac{Sol_G(x)}{K}.$$
On the other hand, $Sol_G(x)/K$ is a subgroup of $G/K$ with order
$2$. It implies that $Sol_{G/K}(xK)=\langle xK\rangle$. We show
that $G/K$ is an abelian simple group. To do this, we suppose that
$G/K$ is not simple. Therefore, there exists a nontrivial normal
subgroup $N/K$ of $G/K$. Assume first that $xK\in N/K$. Since
$Sol_{G/K}(xK)$ is the union of all solvable subgroups of $G/K$
containing $xK$ and $Sol_{G/K}(xK)=\langle xK\rangle$, so it is
seen that $xK$ is a sylow $2$-subgroup of $G/K$. Now, we use
Frattini's argument and obtain that $G/K=N_{G/K}(\langle
xK\rangle)N/K$. Moreover,
$$\langle xK\rangle\subseteq N_{G/K}(\langle xK\rangle)\subseteq
Sol_{G/K}(xK)$$ which yields that $N_{G/K}(\langle
xK\rangle)=\langle xK\rangle$. Thus $G/K=N/K$, that is impossible.
Therefore, we may suppose that $xK\notin N/K$. It is clear that
there exists a prime $r$ dividing $N/K$. If $R/K$ is a sylow
$r$-subgroup of $N/K$, then we can see from Frattini's argument
that
$$G/K=N_{G/K}(R/K) N/K.$$ By assumption, we have $xK\notin N/K$.
Note that $o(xK)=2$ and so we can not write $xK$ as product of two
nontrivial elements $g_1K\in N_{G/K}(R/K)$ and $g_2K\in N/K$. It
forces that $xK\in N_{G/K}(R/K)$. It implies that $\langle
xK\rangle R/K$ is a solvable subgroup of $G/K$ containing $\langle
xK\rangle$ while $\langle xK\rangle$ is the largest solvable
subgroup of $G/K$ having element $xK$. So we derive a
contradiction. It follows that $G/K$ is a simple group.

As before, $\langle xK \rangle$ is a sylow $2$-subgroup of $G/K$
with order $2$. It forces that $G/K$ is not a non-abelian simple
group. It follows that $|G/K|=2$ and so $|G|=4$ which is false.

We conclude that $x^2\notin Sol(G)$. Consequently, $x^2, x^3$ are
not adjacent to $x$. Hence, $deg(x)\leqslant n-3$ that is
impossible.

We deduce that $|Sol(G)|=1$. As mentioned above,
$|Sol_G(x)|=|Sol(G)|+2$ and thus $|Sol_G(x)|=3$. In a view of
Lemma \ref{sol} $(2)$, $|Sol_G(x)|$ is divisible by $o(x)$ and
hence $o(x)=3$. We observe that $$\langle x\rangle\subseteq
N_G(\langle x\rangle)\subseteq Sol_G(x).$$ By a similar way, we
get that $Sol_G(x)=\langle x\rangle$. So we conclude
$$\langle x\rangle=N_G(\langle x\rangle)=Sol_G(x).$$ It follows
that there is no solvable subgroup of $G$ containing $x$ except
for $\langle x\rangle$. Clearly, if $R$ is a sylow $3$-subgroup of
$G$, then $|R|=3$. To gain a contradiction, we will try to find a
solvable subgroup of $G$ containing $x$ distinct from $\langle
x\rangle$.

Assume first that $G$ is not a simple group. Then it has a
nontrivial normal subgroup, say $N$. Suppose that $x\in N$. Thus,
we obtain from Frattini's argument that $$G=N_G(\langle x\rangle
)N=\langle x\rangle N=N,$$ which is false. It implies that
$x\notin N$. Since $N\neq 1$, so there exists a prime $p$ dividing
$|N|$. Let $P$ be a sylow $p$-subgroup of $N$. Again, by
Frattini's argument, we find $G=N_G(P)N$. According to assumption,
we have $x\notin N$. Since $o(x)=3$, hence we can not write $x$ as
product of two nontrivial elements $g_1\in N_G(P)$ and $g_2\in N$.
It forces that $x\in N_G(P)$. Therefore, $P\langle x\rangle$ is a
solvable subgroup which is desired.

Next, suppose that $G$ is a non-abelian simple group. Considering
the classification of finite groups, the possibilities for simple
group $G$ are as follows:
\begin{itemize}
 \item[$(1)$] an alternating group $\Bbb A_n$ on $n\geqslant 5$
 letters;
 \item[$(2)$] one of the $26$ sporadic groups;
  \item[$(3)$] a simple group of Lie type.
\end{itemize}
It is worth to mention that the order of a  sylow $3$-subgroup of
$G$ is $3$.

If $G$ is an alternating or sporadic group, then according to the
order of these groups, $G$ is one of groups $\Bbb A_5$ and $J_1$.
It is seen from \cite{atlas} that if $x\in \Bbb A_5$ (resp.
 $J_1$) with $o(x)=3$, then $x$ is included in some solvable
 subgroups of $\Bbb A_5$ distinct from $\langle
x\rangle$(resp. $J_1$).

 Let now $G$ be a simple group of Lie type. Using the
 orders of Lie type groups, it is enough to examine the following
 groups:
\begin{itemize}
\item the projective special linear groups $A_1(3)$ and $A_1(q)$
defined over a field of characteristic $p$;  \item $A_2(q)$ where
$3\mid {q+1}$ and $9\nmid {q+1}$; \item the unitary group
$^2A_2(q)$ where $3\mid {q-1}$ and $9\nmid {q-1}$.
\end{itemize}
It is good to note that the structure of all subgroups of $A_1(q)$
are determined in \cite{suzuki}. Moreover, using Tables $8. 3$,
$8. 5$ in \cite{baray}, we can find the maximal subgroups of
$A_2(q)$ and $^2A_2(q)$. Thus, it is easily seen that if $x$ is an
element of one of these groups with $o(x)=3$, then $x$ is included
in some solvable subgroups distinct from $\langle x\rangle$.

Therefore, the proof is complete.
\end{proof}

For a finite group $G$, we define
${\rm Ord}(Sol_G)=\{|Sol_G(x)| \ | \ x\in G\}$. Now, it can be asked the following question.
\begin{problem}
Let $G$ and $H$ be two finite groups. If ${\rm Ord}(Sol_G)$
coincides with ${\rm Ord}(Sol_H)$, then is $G$ isomorphic to $H$?
\end{problem}

\end{document}